\documentclass[twoside,leqno, british,12pt]
{amsart}%{article}
\usepackage{amsmath, amssymb, amsbsy, amsthm}
\usepackage{eucal,url,amssymb,%stmaryrd,%booktabs,%
enumerate,amscd,stmaryrd%paralist
}

\usepackage[T1]{fontenc}
\usepackage[british]{babel}
\usepackage[pagebackref,colorlinks=true]{hyperref}  %use this for links in pdf file

\begin{document}

\begin{abstract}
A Sasaki-like almost contact complex Riemannian manifold is
defined as an almost contact complex Riemannian manifold which
complex cone is a holomorphic complex Riemannian manifold.
Explicit compact and non-compact examples are given. A
canonical construction producing a Sasaki-like almost contact
complex Riemannian manifold from a holomorphic complex
Riemannian manifold is presented and called an $S^1$-solvable
extension.
\end{abstract}


\title{Sasaki-like almost contact complex Riemannian manifolds}

\author[S. Ivanov, H. Manev, M. Manev]{Stefan Ivanov, Hristo Manev, Mancho Manev}
\address[S. Ivanov]{University of Sofia, Faculty of Mathematics and
Informatics, 5 James Bourchier Blvd, 1164 Sofia, Bulgaria}
\address{and Institute of Mathematics and Informatics, Bulgarian Academy of
Sciences} \email{ivanovsp@fmi.uni-sofia.bg}
%\author{Mancho Manev}
%\author{Hristo Manev}
\address[H. Manev, M. Manev]{University of Plovdiv, Faculty of
Mathematics and Informatics, 236 Bulgaria Blvd, 4027 Plovdiv,
Bulgaria} \email{hmanev@uni-plovdiv.bg, mmanev@uni-plovdiv.bg}

%%% Theorem Dike Envirouments

%%%% Local Definitions start here
\frenchspacing

\newcommand{\ie}{i.e. }
\newcommand{\X}{\mathfrak{X}}
\newcommand{\F}{\mathcal{F}}
\newcommand{\M}{(M,\allowbreak{}\f,\allowbreak{}\xi,\allowbreak{}\eta,\allowbreak{}g)}
\newcommand{\R}{\mathbb{R}}
\newcommand{\C}{\mathbb{C}}
\newcommand{\n}{\nabla}
\newcommand{\f}{\varphi}
\newcommand{\ta}{\theta}
\newcommand{\lm}{\lambda}
\newcommand{\om}{\omega}
\newcommand{\Span}{\mathrm{span}}
\newcommand{\Id}{\mathrm{Id}}
\newcommand{\D}{\mathrm{d}}
\newcommand{\ddr}{\tfrac{\D}{\D r}}
\newcommand{\g}{\check{g}}
\newcommand{\nn}{\check{\n}}
\newcommand{\ddu}[1]{\frac{\partial}{\partial u^{#1}}}
\newcommand{\ddv}[1]{\frac{\partial}{\partial v^{#1}}}
\newcommand{\ddt}{\tfrac{\D}{\D t}}

\newcommand{\thmref}[1]{The\-o\-rem~\ref{#1}}
\newcommand{\propref}[1]{Pro\-po\-si\-ti\-on~\ref{#1}}
\newcommand{\secref}[1]{\S\ref{#1}}
\newcommand{\lemref}[1]{Lem\-ma~\ref{#1}}
\newcommand{\dfnref}[1]{De\-fi\-ni\-ti\-on~\ref{#1}}
\newcommand{\corref}[1]{Corollary~\ref{#1}}
%\newcommand{\eqref}[1]{(\ref{#1})}

%\renewcommand{\thefootnote}{\fnsymbol{footnote}}

% THEOREM Environments ---------------------------------------------------
%\newtheorem{thm}{Theorem}
%\newtheorem{lem}[thm]{Lemma}
%\newtheorem{prop}[thm]{Proposition}
%\newtheorem{cor}[thm]{Corollary}
%\newdefinition{rmk}{Remark}
%\newdefinition{ack}{Acknowledgements}
%\newproof{pf}{Proof}
%\newproof{pot}{Proof of Theorem \ref{thm-geom}}

%\numberwithin{equation}{section}
\newtheorem{thm}{Theorem}[section]
\newtheorem{lem}[thm]{Lemma}
\newtheorem{prop}[thm]{Proposition}
\newtheorem{cor}[thm]{Corollary}
\newtheorem{conv}[thm]{Convention}
\newtheorem{rmk}[thm]{Remark}

\theoremstyle{definition}
\newtheorem{defn}{Definition}[section]

\hyphenation{Her-mi-ti-an ma-ni-fold ah-ler-ian}

%\keywords{ }

%\subjclass[2000]{} %

%%%% End of Local Definitions
%{\small
%\begin{abstract}
%\end{abstract}

%%% ----------------------------------------------------------------------
\maketitle
%%% ----------------------------------------------------------------------

\begin{center}
\date{\today}
\end{center}

\setcounter{tocdepth}{2} \tableofcontents

%\vspace{1cm}

\section{Introduction}
The almost contact complex Riemannian manifold is an
odd-dimensional pseudo-Riemannian manifold equipped with a 1-form
$\eta$ and a codimension one distribution $H=Ker(\eta)$ endowed
with a complex Riemannian structure. More precisely, the
$2n$-dimensional distribution $H$ is equipped with a pair
consisting of an almost complex structure and a pseudo-Riemannian
metric of signature $(n,n)$ compatible in the way that  the almost
complex structure acts as an anti-isometry on the metric. Almost
contact complex Riemannian manifolds are investigated and studied
in \cite{GaMiGr,Man4,Man31,ManGri1,ManGri2,MI, MI1,NakGri2}.

The main goal of this note is to find  a  class of almost contact
complex  Riemannian manifolds resemble some basic properties of
the well known Sasakian manifolds. We define the class of
Sasaki-like spaces as an almost contact complex Riemannian
manifold which complex cone is a holomorphic complex Riemannian
manifold. We note that a holomorphic complex Riemannian manifold
is a complex manifold endowed with a complex Riemannian metric
whose local components in holomorphic coordinates are holomorphic
functions (see \cite{Manin}). We  determine the Sasaki-like almost
contact complex Riemannian  structure  with an explicit expression
of the covariant derivative of the structure tensors (cf.
Theorem~\ref{ssss}) and construct explicit compact and non-compact
examples. We also present a canonical construction producing a
Sasaki-like almost contact complex Riemannian manifold from any
holomorphic complex Riemannian manifold  which we called an
$S^1$-\emph{solvable extension} (cf. Theorem~\ref{ext}). Studying
the curvature of Sasaki-like spaces we show that it is completely
determined by the curvature of the underlying holomorphic complex
Riemannian manifold. We develop  gauge transformations of
Sasaki-like spaces, i.e. we  find the class of contact conformal
transformations of an almost contact complex Riemannian manifolds
which preserve the Sasaki-like condition.

\begin{conv}%\hfill\break\vspace{-15pt} %
Let $\M$ be a $(2n+1)$-dimensional almost contact complex
Riemannian manifold with a pseudo-Riemannian metric $g$ of
signature $(n+1,n)$.
\begin{enumerate}[a)]
\item[\emph{a)}] %We shall use  $A,B,C$ to denote smooth vector fields on
%$M$, $A,B,C\in\X(M)$.
We shall use  $x$, $y$, $z$, $u$ to denote smooth vector fields on
$M$, \ie $x,y,z,u\in\X(M)$.

\item[\emph{b)}] We shall use $X$, $Y$, $Z$, $U$
to denote smooth horizontal vector fields on $M$, i.e. $%
X,Y,Z,U\in H=\ker(\eta)$.

\item[\emph{c)}]  The $2n$-tuple
$\{e_1,\dots,e_{n},e_{n+1}=\f e_1,\dots,e_{2n}=\f e_n\}$ denotes a
local orthonormal basis of the horizontal space $H$.

\item[\emph{d)}] For an orthonormal basis
$
\{e_0=\xi,e_1,\dots,e_{n},e_{n+1}=\f e_1,\dots,e_{2n}=\f e_n\}
$ we
denote $\varepsilon_i=\mathrm{sign}(g(e_i,e_i))=\pm 1$, where
$\varepsilon_i=1$ for $i=0,1,\dots,n$ and $\varepsilon_i=-1$ for
$i=n+1,\dots,2n$.
\end{enumerate}
\end{conv}

%%%%%%%%%%%%%%%%%%%%%%%%%%%%%%%%%%

\section{Almost contact complex Riemannian manifolds}

Let $(M,\f,\xi,\eta)$ be an almost contact manifold,  i.e.  $M$ is
a $(2n+1)$-dimensional differen\-tia\-ble manifold with an almost
contact structure $(\f,\xi,\eta)$ consisting of an endomorphism
$\f$ of the tangent bundle, a vector field $\xi$ and its dual
1-form $\eta$ such that the following algebraic relations are
satisfied:
\begin{equation}\label{str1}
\f\xi = 0,\quad \f^2 = -\Id + \eta \otimes \xi,\quad
\eta\circ\f=0,\quad \eta(\xi)=1.
\end{equation}
An almost contact structure $(\f,\xi,\eta)$ on $M$ is called
\emph{normal} and respectively $(M,\f,\allowbreak{}\xi,\eta)$ is a
\emph{normal almost contact manifold} if the corresponding almost
complex structure $\check J$  on $M'=M\times \R$ defined by
\begin{equation}\label{concom}
\check JX=\f X,\qquad \check J\xi=r\ddr,\qquad \check J\ddr=-\tfrac{1}{r}\xi
\end{equation}
 is integrable (i.e.
$M'$ is a complex manifold) \cite{SaHa}. The almost contact
structure is normal if and only if the Nijenhuis tensor of
$(\f,\xi,\eta)$ is zero \cite{Blair}. The Nijenhuis tensor $N$ of
the almost contact structure is defined by
$$ N = [\f, \f]+\D{\eta}\otimes\xi,\quad
[\f, \f](x, y)=\left[\f x,\f
y\right]+\f^2\left[x,y\right]-\f\left[\f x,y\right]-\f\left[x,\f
y\right],$$
where $[\f,\f]$ is the Nijenhuis torsion of $\f$.

Let the almost contact manifold $(M,\f,\xi,\eta)$ be endowed with
a pseudo-Riemann\-ian metric $g$ of signature $(n+1,n)$ which is
compatible with the almost contact structure in the following way
\[
g(\f x, \f y ) = - g(x, y ) + \eta(x)\eta(y).
\]
The associated metric $\widetilde{g}$ of $g$ on $M$ is defined by
$\widetilde{g}(x,y)=g(x,\f y)\allowbreak+\eta(x)\eta(y)$. Both
metrics $g$ and $\widetilde{g}$  are necessarily of signature
$(n+1,n)$.

The manifold  $(M,\f,\xi,\eta,g)$ is known as an almost contact
manifold with B-metric or an \emph{almost contact B-metric
manifold} \cite{GaMiGr}. The manifold
$(M,\f,\xi,\eta,\widetilde{g})$ is also an almost contact B-metric
manifold. We will call these manifolds \emph{almost contact
complex Riemannian manifolds}.

The structure group of the almost contact complex Riemannian
manifolds is
$O(n,\mathbb{C})\allowbreak\times
I_1=\allowbreak\bigl(GL(n,\mathbb{C})\cap
O(n,n)\bigr)\allowbreak\times I_1$, i.e. it consists of real
square matrices of order $2n+1$ of the following type
\[
\left(%
\begin{array}{r|c|c}
  A & B & \vartheta^T\\ \hline
  -B & A & \vartheta^T\\ \hline
  \vartheta & \vartheta & 1 \\
\end{array}%
\right),\qquad %
\begin{array}{l}
  A^TA-B^TB=I_n,\\%
  B^TA+A^TB=O_n,
\end{array}%
\quad A, B\in GL(n;\mathbb{R}),
\]
where $\vartheta$ and its transpose $\vartheta^T$ are the zero row
$n$-vector and the zero column $n$-vector; $I_n$ and $O_n$ are the
unit matrix and the zero matrix of size $n$, respectively.

%\subsection{The structural tensor $F$}

The covariant derivatives of $\f$, $\xi$, $\eta$ with respect to
the Levi-Civita connection $\n$ play a fundamental role in the
differential geometry on the almost contact manifolds.  The
structure tensor $F$ of type (0,3) on $\M$ is defined by
\begin{equation}\label{F=nfi}
F(x,y,z)=g\bigl( \left( \nabla_x \f \right)y,z\bigr).
\end{equation}
It has the following properties \cite{GaMiGr}:
\begin{equation}\label{F-prop}
\begin{split}
F(x,y,z)&=F(x,z,y)=F(x,\f y,\f z)+\eta(y)F(x,\xi,z)
+\eta(z)F(x,y,\xi).
\end{split}
\end{equation}
The relations of $\n\xi$ and $\n\eta$ with $F$ are:
\begin{equation}\label{Fxieta}
    \left(\n_x\eta\right)y=g\left(\n_x\xi,y\right)=F(x,\f y,\xi).
\end{equation}

The following 1-forms  associated with $F$:
\[
\ta(z)=\sum_{i=1}^{2n}\varepsilon_iF(e_i,e_i,z),\quad
\ta^*(z)=\sum_{i=1}^{2n}\varepsilon_iF(e_i,\f e_i,z)
\]
satisfy the obvious relation $\ta^*\circ\f=-\ta\circ\f^2$.

Besides the Nijenhuis tensor $N$, the following symmetric
(1,2)-tensor $\widehat N$ is defined in \cite{MI1} as follows:
consider the symmetric brackets $\{x,y \}$ given by
\[
g(\{x,y\},z)=g(\nabla_xy+\nabla_yx,z)=xg(y,z)+yg(x,z)-zg(x,y)-g([y,z],x)+g([z,x],y);
\]
set
\[
\{\f ,\f\}(x,y)=\{\f x,\f y\}+\f^2\{x,y\}-\f\{\f x,y\}-\f\{x,\f
y\}
\]
and define the symmetric tensor $\widehat N$ as follows
\cite{MI1}
\[
\widehat N=\{\f,\f\}+(\mathcal L_{\xi}g)\otimes \xi,
\]
where $\mathcal L$ denotes the Lie derivative.
The tensor $\widehat N$ is also called  \emph{the associated
Nijenhuis tensor}.

We define the corresponding tensors of type (0,3) by the same
letter, $N(x,y,z)=g(N(x,y),z)$, $\widehat N(x,y,z)=g(\widehat
N(x,y),z)$. Both tensors $N$ and $\widehat N$ can be expressed in
terms of the fundamental tensor $F$ as follows \cite{MI1}
\begin{gather}
\begin{array}{l}
N(x,y,z)=F(\f x,y,z)-F(\f y,x,z)-F(x,y,\f z)+F(y,x,\f z)
\label{enu}\\[4pt]
\phantom{N(x,y,z)=}+\eta(z)\bigl[F(x,\f y,\xi)-F(y,\f
x,\xi)\bigr],
\end{array}\\[4pt]
\begin{array}{l}
\widehat N(x,y,z)=F(\f x,y,z)+F(\f y,x,z)-F(x,y,\f z)-F(y,x,\f
z)\\[4pt]
\phantom{\widehat N(x,y,z)=}+\eta(z)\bigl[F(x,\f y,\xi)+F(y,\f
x,\xi)\bigr].\label{enhat}
\end{array}
\end{gather}

\subsection{Relation with holomorphic complex Riemannian manifolds}
Let us remark that the $2n$-dimensional  distribution
$H=\ker(\eta)$  is endowed with an almost complex structure
$J=\f|_H$ and a metric $h=g|_H$, where $\f|_H$, $g|_H$ are  the
restrictions of $\f$, $g$ on $H$, respectively, and the metric $h$
is compatible with $J$ as follows
\begin{equation}\label{norden}
h(JX,JY)=-h(X,Y) ,\quad \widetilde{h}(X,Y):=h(X,JY).
\end{equation}
The distribution  $H$ can be considered as an $n$-dimensional
complex Riemannian distribution with a complex Riemannian metric
$g^{\C}=h+i\widetilde h=g|_H+i\widetilde{g}|_H$.

We recall that a $2n$-dimensional almost complex manifold
$(N,J,h)$ endowed with a pseudo-Riemannian metric of signature
$(n,n)$ satisfying \eqref{norden} is known as an almost complex
manifold with Norden metric \cite{N1,N2,v,GGM0,KS,KO},  an almost
complex manifold with B-metric \cite{GGM,GM}  or an almost complex
manifold with complex Riemannian metric \cite{LeB,Manin,GaIv,BFF,Low}.
When the almost complex structure $J$ is parallel with respect to
the Levi-Civita connection $\nabla^h$ of the metric $h$,
$\nabla^hJ=0$, then the manifold is known as a K\"ahler-Norden
manifold, a K\"ahler manifold with B-metric or a holomorphic
complex Riemannian manifold. In this case  the almost complex
structure $J$ is integrable and the local components of the complex metric
in holomorphic coordinate system are holomorphic functions.
A four-dimensional example of a K\"ahler manifold with Norden metric
has been given in \cite{N1}, another approach to the
K\"ahler manifolds with Norden metric has been used in \cite{N2}
and in \cite{v}, there has been proved that the four-dimensional
sphere of Kotel'nikov-Study carries a structure of a K\"ahler
manifold with Norden metric.

\subsection{The  case of parallel structures} The simplest case  of
almost contact complex Riemannian  manifolds  is when the
structures are $\n$-parallel, $\n\f=\n\xi=\n\eta=\n g=\n
\widetilde{g}=0$, and it is determined by the condition
$F(x,y,z)=0$. In this case the distribution $H$ is involutive. The
corresponding integral submanifold is a totally geodesic
submanifold which inherits a holomorphic complex Riemannian
structure and the almost contact complex Riemannian manifold is
locally  a pseudo-Riemannian product of a holomorphic complex
Riemannian manifold with a real interval.

\section{Sasaki-like almost contact complex Riemannian manifolds}
In this section we consider the complex Riemannian cone over an
almost contact complex Riemannian manifold and determine a
Sasaki-like almost contact complex Riemannian manifold with the
condition that its complex Riemannian cone is a holomorphic complex
Riemannian manifold.

\subsection{Holomorphic complex Riemannian cone}
Let $\M$ be an almost contact Riemannian manifold of dimension
$2n+1$. We consider the cone over $M$   %$$$\mathfrak{Cone}$
%$\mathbb Cone$ %
$\mathcal{C}(M)=M\times \R^-$ with the  almost complex structure
determined in  \eqref{concom} and the complex Riemannian metric
defined by
\begin{equation}\label{barg}
\check{g}\left(\left(x,a\ddr\right),\left(y,b\ddr\right)\right)
=r^2g(x,y)+\eta(x)\eta(y)-ab,
\end{equation}
where $r$ is the coordinate on $\R^-$ and $a$, $b$ are
$C^{\infty}$ functions on $M\times \R^-$.

Using the general Koszul formula
\begin{multline}\label{koszul}
2g(\nabla_xy,z)=xg(y,z)+yg(z,x)-zg(x,y)\\%[4pt]
%\phantom{g(\nabla_xy,z)=\frac12\bigl[}
+g([x,y],z)-g([y,z],x)+g([z,x],y),
\end{multline}
we calculate from \eqref{barg} that the  non-zero components of
the Levi-Civita connection $\nn$ of the complex Riemannian metric
$\check g$ on $\mathcal{C}(M)$ are given by
\[
\begin{array}{l}
    \g\left(\nn_X Y,Z\right)=r^2 g\left(\n_X Y,Z\right),\qquad
    \g\left(\nn_X Y,\ddr\right)=-r g\left(X, Y\right), %
    \\[4pt]
    \g\left(\nn_X Y,\xi\right)=r^2 g\left(\n_X
    Y,\xi\right)+\tfrac{1}{2}\left(r^2-1\right)\D\eta(X,Y),
\\[4pt]
    \g\left(\nn_X \xi,Z\right)=r^2 g\left(\n_X \xi,
    Z\right)-\tfrac{1}{2}\left(r^2-1\right)\D\eta(X,Z),
    \\[4pt]
    \g\left(\nn_{\xi} Y,Z\right)=
    r^2 g\left(\n_{\xi}Y,Z\right)-\tfrac{1}{2}(r^2-1)\D\eta(Y,Z),
    \\[4pt]
    \g\left(\nn_{\xi} Y,\xi\right)= g\left(\n_{\xi}Y,\xi\right), \qquad
    \g\left(\nn_{\xi} \xi,Z\right)= g\left(\n_{\xi}\xi,Z\right),
    \\[4pt]
    \g\left(\nn_X \ddr,Z\right)=r g\left(X,Z\right), \qquad
    \g\left(\nn_{\ddr} Y,Z\right)=r g\left(Y,Z\right).
\end{array}
\]
Applying \eqref{concom} we calculate from the formulas above  that
the non-zero components of the covariant derivative $\nn \check J$
of the almost complex structure $\check J$ are given by
\[
\begin{array}{rl}
    &\g\left(\left(\nn_X \check J\right)Y,Z\right)
    =r^2 g\left(\left(\n_X \f\right)Y,Z\right),
    \\[4pt]
    &\g\left(\left(\nn_X  \check J\right)Y,\xi\right)
    =r^2 \left\{g\left(\left(\n_X \f\right)Y,\xi\right)
    +g(X,Y)\right\}+\tfrac{1}{2}\left(r^2-1\right)\D\eta(X,\f Y),%
    \\[4pt]
%    &\phantom{\g\left(\left(\nn_X  \check J\right)Y,\xi\right)=}
%    +\frac{1}{2}\left(r^2-1\right)\D\eta(X,\f Y),
%   \nonumber\\[4pt]
    &\g\left(\left(\nn_X  \check J\right)Y,\ddr\right)=-r\left\{g\left(\n_X \xi, Y\right)
    +g\left(X,\f Y\right)\right\}+\tfrac{1}{2r^2}(r^2-1)\D\eta(X,Y),%
\\[4pt]
    &\g\left(\left(\nn_X  \check J\right)\xi,Z\right)=-r^2\left\{g\left(\n_X \xi,\f Y\right)
    -g\left(X,Z\right)\right\}+\tfrac{1}{2}(r^2-1)\D\eta(X,\f Z),
    \\[4pt]
    &\g\left(\left(\nn_X  \check J\right)\ddr,Z\right)=-r\left\{g\left(\n_X \xi, Z\right)
    +g\left(X,\f Z\right)\right\}+\tfrac{1}{2r}(r^2-1)\D\eta(X,Z),
    \\[4pt]
    &\g\left(\left(\nn_{\xi}  \check J\right)Y,Z\right)
    =r^2g\left(\left(\n_{\xi}\f\right)Y,Z\right)
    -\tfrac{1}{2}(r^2-1)\left\{\D\eta(\f Y, Z)-\D\eta(Y,\f Z)\right\},
    \\[4pt]
    &\g\left(\left(\nn_{\xi}  \check J\right)Y,\xi\right)=-g(\n_{\xi}\xi,\f Y), \qquad
    \g\left(\left(\nn_{\xi}  \check J\right)\xi,Z\right)=-g\left(\n_{\xi}\xi, \f Z\right),
    \\[4pt]
    &\g\left(\left(\nn_{\xi}  \check J\right)Y,\ddr\right)=
    -\tfrac{1}{r}g\left(\n_{\xi}\xi, Y\right), \qquad  \g\left(\left(\nn_{\xi}
    \check J\right)\ddr,Z\right)=-\tfrac{1}{r}g\left(\n_{\xi}\xi, Z\right).
\end{array}
\]
We have
\begin{prop}\label{defs}
The complex Riemannian cone $\mathcal{C}(M)$ over an almost
contact complex Riemannian manifold $\M$  is a holomorphic complex
Riemannian space if and only if the following conditions hold
\begin{alignat}{1}
&F(X,Y,Z)=F(\xi,Y,Z)=F(\xi,\xi,Z)=0\label{sasaki},\\[4pt]
&F(X,Y,\xi)=-g(X,Y)\label{sasaki0}.%,\\[4pt]
%    \left(\n_X\eta\right)Y&=-g(\f X, Y)=g(  \n_X
%\xi,Y)\label{sasaki1}.
\end{alignat}
\end{prop}

\begin{proof}
We obtain from the expressions above that the complex Riemannian
cone $(\mathcal{C}(M),\check J,\check g)$ is a holomorphic Riemannian
manifold (a K\"ahler manifold with Norden metric), \ie $\nn \check J=0$,
if and only if  the almost contact complex Riemannian manifold
$\M$ satisfies the following conditions
\begin{alignat}{1}
    &F(X,Y,Z)=0,
    \nonumber\\[4pt]
    &F(X,Y,\xi)=-g(X,Y)-\tfrac{1}{2r^2}\left(r^2-1\right)\D\eta(X,\f Y)\label{FXYxi},\\[4pt]
%    &F(X,\f Y,\xi)=-g(X,\f
%    Y)+\tfrac{1}{2r^2}\left(r^2-1\right)\D\eta(X,Y),\label{FXYxi}\\[4pt]
    &F(\xi,Y,Z)=\tfrac{1}{2r^2}\left(r^2-1\right)\left\{\D\eta(\f Y,Z)
    -\D\eta(Y,\f Z)\right\},\label{FxiYZ}\\[4pt]
    &F(\xi,\xi,Z)=0, \qquad \n_{\xi}\xi=0\label{nxixi}.
\end{alignat}
The condition $\nn \check J=0$ implies the integrability of $\check J$,
hence the structure is normal.
%The equality  \eqref{nxixi} shows the that the integral curves of $\xi$ are geodesics.

Further, according to \eqref{FXYxi}, we get
\begin{equation}\label{nablaeta}
\left(\n_X\eta\right)Y=-g(X,\f
Y)+\tfrac{1}{r^2}\left(r^2-1\right)\D\eta(X,Y),
\end{equation}
yielding
$\D\eta(X,Y)=\tfrac{1}{2r^2}\left(r^2-1\right)\D\eta(X,Y)$ since
the metric $\widetilde g$ is symmetric. The latter equality  shows
$\D\eta(X,Y)=0$ which together with \eqref{nablaeta} yields
\begin{equation}\label{sasaki1}
\left(\n_X\eta\right)Y=-g(X,\f  Y).
\end{equation}
From \eqref{nxixi} we get
$\D\eta(\xi,X)=(\n_{\xi}\eta)(X)-(\n_X\eta)(\xi)=0$. Hence,
$\D\eta=0$. Substitute $\D\eta=0$ into
\eqref{FXYxi}-\eqref{FxiYZ} to complete the proof of the
proposition.
\end{proof}

\begin{defn}
An almost contact complex Riemannian manifold $\M$  is said to be
\emph{Sasaki-like} if the structure tensors $\f, \xi, \eta, g$
satisfy the equalities \eqref{sasaki} and \eqref{sasaki0}.
\end{defn}

To characterize the Sasaki-like almost contact complex Riemannian
manifold by the structure tensors, we need the next general
result
\begin{thm}\label{gen}
Let $\M$ be an almost contact complex Riemannian manifold.
Then the covariant derivative of $\f$ is given by the formula
\begin{equation}\label{nabf}
\begin{split}
g(\n_x\f)y,z)&=F(x,y,z)\\
&=-\frac14\bigl[N(\f x,y,z)+N(\f x,z,y)
+\widehat N(\f x,y,z)+\widehat N(\f x,z,y)\bigr]\\
&\phantom{=\ }+\frac12\eta(x)\bigl[N(\xi,y,\f z)+\widehat
N(\xi,y,\f z)+\eta(z)\widehat N(\xi,\xi,\f y)\bigr].
\end{split}
\end{equation}
\end{thm}
\begin{proof}
Taking the sum of \eqref{enu} and \eqref{enhat}, we obtain
\begin{equation}\label{sum1}
F(\f x,y,z)-F(x,y,\f z)=\frac12\bigl[N(x,y,z)+\widehat
N(x,y,z)\bigr]-\eta(z)F(x,\f y,\xi).
\end{equation}
The  identities  \eqref{F-prop} together with \eqref{str1} imply
\begin{equation}\label{sum2}
F(x,y,\f z)+F(x,z,\f y)=\eta(z)F(x,\f y,\xi)+\eta(y)F(x,\f z,\xi).
\end{equation}
A suitable combination of \eqref{sum1} and \eqref{sum2} yields
\begin{equation}\label{nabff}
F(\f x,y,z)=\frac14\bigl[N(x,y,z)+N(x,z,y)+\widehat
N(x,y,z)+\widehat N(x,z,y)\bigr].
\end{equation}
Applying \eqref{str1}, we obtain from \eqref{nabff}
\begin{multline}\label{nabff1}
F(x,y,z)=\eta(x)F(\xi,y,z)\\-\frac14\bigl[N(\f x,y,z)+N(\f
x,z,y)+\widehat N(\f x,y,z)+\widehat N(\f x,z,y)\bigr].
\end{multline}
Set $x=\xi$ and $z\rightarrow \f z$ into \eqref{sum1} and use \eqref{str1} to get
\begin{equation}\label{nabff2}
F(\xi,y,z)=\frac12\bigl[N(\xi,y,\f z)+\widehat N(\xi,y,\f
z)\bigr]+\eta(z)F(\xi,\xi,y).
\end{equation}
Finally, set $y=\xi$ into \eqref{nabff2} and use the general
identities $N(\xi,\xi)=F(\xi,\xi,\xi)=0$ to obtain
\begin{equation}\label{fff}F(\xi,\xi,z)=\frac12\widehat N(\xi,\xi,\f z).
\end{equation}
Substitute \eqref{fff} into \eqref{nabff2} and the obtained identity insert
into \eqref{nabff1} to get \eqref{nabf}.
\end{proof}

The next result determines  the Sasaki-like spaces by the structure tensors.

\begin{thm}\label{ssss}
Let $\M$ be an almost contact complex Riemannian manifold.
The following conditions are equivalent:
\begin{itemize}
\item[a)] The manifold $\M$ is a Sasaki-like almost contact
complex Riemannian manifold; %
\item[b)] The covariant derivative
$\nabla \f$ satisfies the equality
\begin{equation}\label{defsl}
\begin{array}{l}
(\nabla_x\f)y=-g(x,y)\xi-\eta(y)x+2\eta(x)\eta(y)\xi;
%\\[4pt] \phantom{(\nabla_x\f)y}=g(\f x,\f y)\xi+\eta(y)\f^2 x;
\end{array}
\end{equation}
\item[c)] The Nijenhuis tensors $N$ and $\widehat N$ satisfy the
relations:
\begin{equation}\label{sasnn}
N=0,\qquad  \widehat
N=-4\left(\widetilde{g}-\eta\otimes\eta\right)\otimes\xi.
\end{equation}
\end{itemize}
\end{thm}

\begin{proof}
It is easy to check using \eqref{Fxieta} and \eqref{F-prop}
that \eqref{defsl} is equivalent to the system of the equations
\eqref{sasaki} and \eqref{sasaki0} which established the equivalence between
a) and b) in view of Proposition~\ref{defs}.

Substitute  \eqref{defsl} consequently into \eqref{enu} and
\eqref{enhat} to get \eqref{sasnn} which gives the implication
$b)\Rightarrow c)$.

Now, suppose \eqref{sasnn} holds. Consequently, we obtain
$\widehat N(\xi,y)=0$. Now, \eqref{defsl} follows with a
substitution of the last equality together with  \eqref{sasnn}
into \eqref{nabf} which completes the proof.
\end{proof}

\begin{cor}
Let $\M$ be a Sasaki-like almost contact complex Riemannian
manifold. Then we have %the next conditions hold
\begin{itemize}
\item[a)] the manifold $\M$ is normal, $N=0$, the fundamental
1-form $\eta$ is closed, $\D\eta=0$, and the integral curves of
$\xi$ are
geodesics, $\n_{\xi}\xi=0$; %
\item[b)]  the 1-forms $\theta$ and $\theta^*$ satisfy the
equalities $ \theta=-2n\,\eta$ and $\theta^*=0$.
\end{itemize}
\end{cor}

\subsection{Examples}
In this section we construct a number of examples of Sasaki-like
almost contact complex Riemannian manifolds.

\subsubsection{Example~1.} Consider the solvable Lie group $G$ of
dimension $2n+1$
 with a basis of left-invariant vector fields $\{e_0,\dots, e_{2n}\}$
 defined by the commutators
\begin{equation}\label{com}
[e_0,e_1]=e_{n+1},\; \dots,\; [e_0,e_n]=e_{2n},\;
[e_0,e_{n+1}]=-e_1,\; \dots,\; [e_0,e_{2n}]=-e_n.
\end{equation}
Define an invariant almost contact complex Riemannian structure
%(an almost contact B-metric structure)
 on $G$  by
\begin{equation}\label{strEx1}
%\begin{array}{c}
\begin{array}{rl}
&g(e_0,e_0)=g(e_1,e_1)=\dots=g(e_n,e_n)=1\\[4pt]
&g(e_{n+1},e_{n+1})=\dots=g(e_{2n},e_{2n})=-1,
%\end{array}
\\[4pt]
&g(e_i,e_j)=0,\quad
i,j\in\{0,1,\dots,2n\},\; i\neq j,
\\[4pt]
&\xi=e_0, \quad \f  e_1=e_{n+1},\quad  \dots,\quad  \f e_n=e_{2n}.
\end{array}
\end{equation}

Using the Koszul formula \eqref{koszul} we check that
\eqref{sasaki} and  \eqref{sasaki0} %and \eqref{sasaki1}
are
fulfilled, i.e. this is a Sasaki-like almost contact complex
Riemannian structure.

Let $e^0=\eta$, $e^1$, $\dots$, $e^{2n}$ be the corresponding dual
1-forms, $e^i(e_j)=\delta^i_j$. From \eqref{com} and the formula
for an arbitrary 1-form $\alpha$
\[
\D
\alpha(A,B)=A\alpha(B)-B\alpha(A)-\alpha([A,B]),
\]
it follows that the structure equations of the group are
\begin{equation}\label{comstr}
\begin{array}{llll}
 \D e^0=\D\eta=0,\; &\D e^1=e^{0}\wedge e^{n+1},\;&\dots,\; &\D
e^n=e^{0}\wedge e^{2n},\\[4pt]
&\D e^{n+1}=-e^{0}\wedge e^{1}, &\dots, &\D e^{2n}=- e^{0}\wedge
e^{n}
\end{array}
\end{equation}
and the Sasaki-like almost contact complex Riemannian structure
has the form
\begin{equation}\label{sas}
g=\sum_{i=0}^{2n}\varepsilon_i\left(e^i\right)^2,%
\qquad \f e^1= e^{n+1},\ \dots,\ \f e^n= e^{2n}.
\end{equation}

The group $G$ is the following  \textit{rank-1 solvable
extension of the Abelian group $\R^{2n}$}
\begin{equation}\label{ext1}
\begin{array}{lll}
e^0=\D t, \; &e^1=\cos t\ \D x^1+\sin t\ \D x^{n+1},\;
&e^{n+1}=-\sin
t\ \D x^1+\cos t\ \D x^{n+1},\\[4pt]%\nonumber %
\dots, \; &e^n=\cos t\ \D x^n+\sin t\ \D x^{2n},\; &e^{2n}=-\sin
t\ \D x^n+\cos t\ \D x^{2n}.
\end{array}
\end{equation}
Clearly, the 1-forms defined in \eqref{ext1} satisfy
\eqref{comstr} and the Sasaki-like almost contact complex Riemannian metric has
the form
\begin{equation}\label{sasmetric}
g=\D t^2+\cos{2t}\left(\sum_{i=1}^{2n}\varepsilon_i\left(\D
x^i\right)^2\right)-\sin{2t}\left(-2\sum_{i=1}^n\D x^i\D
x^{n+i}\right).
\end{equation}

It is known that the solvable Lie group $G$ admits a lattice
$\Gamma$ such that the quotient space $G/\Gamma$ is compact (c.f.
\cite[Chapter~3]{OT}). The invariant Sasaki-like almost contact
complex Riemannian structure $(\f, \xi, \eta, g)$  on $G$ descends
to $G/\Gamma$ which supplies a compact Sasaki-like almost contact
complex Riemannian manifold in any dimension.

It follows from \eqref{com}, \eqref{sas}, \eqref{ext1} and  \eqref{sasmetric}
that the distribution
$H=%\ker(\eta)=
\Span\{e_1,\dots,e_{2n}\}$ is integrable and the corresponding
integral submanifold can be considered as the holomorphic complex
Riemannian flat space $\R^{2n}=\Span\{\D x^1,\dots,\D x^{2n}\}$
with the  following holomorphic complex Riemannian  structure  %
\[
J\D x^1=\D x^{n+1},\; \dots,\;
J\D x^n=\D x^{2n};\quad h=\sum_{i=1}^{2n}\varepsilon_i(\D x^i)^2,
\quad \widetilde h =-2\sum_{i=1}^n\D x^i\D x^{n+i}.
\]

\subsubsection{\protect{$S^1$}-solvable extension}
Inspired by Example~1 we proposed the following more general
construction. Let $(N^{2n},J,h,\widetilde{h})$ be a
$2n$-dimen\-sion\-al holomorphic complex Riemannian manifold, i.e.
the almost complex structure $J$ acts as an anti-isometry on the
neutral metric $h$, $h(JX,JY)=-h(X,Y)$ and it is parallel with
respect to the Levi-Civita connection of $h$. In particular, the
almost complex structure $J$ is integrable. The associated neutral
metric $\widetilde{h}$ is defined by $\widetilde{h}(X,Y)=h(JX,Y)$
and it is also parallel with respect to the Levi-Civita connection
of $h$.

Consider the product manifold $M^{2n+1}=\mathbb R^+\times N^{2n}$.
Let $\D t$ be the coordinate 1-form on $\mathbb R^+$ and define an
almost contact complex Riemannian structure on $M^{2n+1}$ as
follows
\begin{equation}\label{strsas}
\eta=\D t, \quad \varphi |_H%{|_{Ker(\eta)}}
=J, \quad \eta\circ\varphi =0,\quad g=\D t^2+\cos{2t}\ h-\sin{2t}\
\widetilde{h}.
\end{equation}

\begin{thm}\label{ext}
Let $(N^{2n},J,h,\widetilde{h})$ be a $2n$-dimensional holomorphic
complex Rie\-mann\-ian manifold. Then the product manifold
$M^{2n+1}=\mathbb R^+\times N^{2n}$ equipped with the almost
contact complex Riemannian structure defined in \eqref{strsas} is
a Sasaki-like almost contact complex Riemannian manifold.

If $N^{2n}$ is compact then $M^{2n+1}=S^1\times N^{2n}$ with the
structure \eqref{strsas} is a compact Sasaki-like almost contact
complex Riemannian manifold.
\end{thm}
\begin{proof}
It is easy to check using \eqref{koszul}, \eqref{strsas} and the
fact that the complex structure $J$ is parallel with respect to
the Levi-Civita connection of $h$ that the structure
defined in \eqref{strsas} satisfies \eqref{sasaki} and
\eqref{sasaki0} and thus $\M$  is a  Sasaki-like almost contact
complex Riemannian manifold.

Now, suppose $N^{2n}$ is a compact holomorphic complex Riemannian
manifold. The equations \eqref{strsas} imply that the metric $g$
is periodic on $\R$ and therefore it descends to the compact
manifold $M^{2n+1}=S^1\times N^{2n}$. Thus we obtain a compact
Sasaki-like almost contact complex Riemannian manifold.
\end{proof}
We call the Sasaki-like almost contact complex Riemannian manifold
constructed in Theorem~\ref{ext} from a holomorphic complex
Riemannian manifold  \emph{an $S^1$-solvable extension of a holomorphic
complex Riemannian manifold}.

\subsubsection{Example~2.} Let us consider the Lie group $G^5$ of
dimension $5$
 with a basis of left-invariant vector fields $\{e_0,\dots, e_{4}\}$
 defined by the commutators
\[
\begin{array}{ll}
[e_0,e_1] = \lm e_2 + e_3 + \mu e_4,\quad &[e_0,e_2] = - \lm e_1 -
\mu e_3 + e_4,\\[4pt]
[e_0,e_3] = - e_1  - \mu e_2 + \lm e_4,\quad &[e_0,e_4] = \mu e_1
- e_2 - \lm e_3,\qquad \lm,\mu\in\R.
\end{array}
\]
Let $G^5$ be equipped with an invariant almost contact complex
Riemannian structure as in \eqref{strEx1} for $n=2$. We calculate
using \eqref{koszul} that the non-zero connection 1-forms of the
Levi-Civita connection are
\[
\begin{array}{c}
\begin{array}{llll}
\n_{e_0} e_1 = \lm e_2 + \mu e_4,\; & \n_{e_1}e_0 = - e_3,\;
& \n_{e_0} e_2 = - \lm e_1 - \mu e_3, \; & \n_{e_2} e_0 = - e_4,\\[4pt]
\n_{e_0} e_3 = - \mu e_2 + \lm e_4,\; & \n_{e_3} e_0 = e_1, \; &
\n_{e_0} e_4 = \mu e_1 - \lm e_3, \; & \n_{e_4}e_0 = e_2,
\end{array}
\\
\begin{array}{c}\\[-8pt]
\n_{e_1}e_3 = \n_{e_2} e_4 = \n_{e_3} e_1 = \n_{e_4}e_2 = - e_0.
\end{array}
\end{array}
\]

Similarly as in Example~1 we verify that the constructed manifold
$(G^5,\f,\xi,\eta,g)$ is a Sasaki-like almost contact complex
Riemannian manifold.

{ Take $\mu=0$ and $\lambda\not=0$. Then the structure equations of the group become
\begin{equation}\label{comstr2}
\begin{array}{llll}
 \D e^0=\D\eta=0,\; &\D e^1=e^{0}\wedge e^3 +\lambda e^0\wedge e^2, &\D
e^2=e^{0}\wedge e^4-\lambda e^0\wedge e^1,\\[4pt]
&\D e^3=-e^{0}\wedge e^{1}+\lambda e^0\wedge e^4, &\D e^4=-
e^{0}\wedge e^2 -\lambda e^0\wedge e^3.
\end{array}
\end{equation}

A basis of 1-forms satisfying \eqref{comstr2} is given by $e^0=\D
t$ and
\[
\begin{split}
e^1=\ &\cos{(1-\lambda)t}\ \D x^1-\cos{(1+\lambda)t}\ \D x^2
+\sin{(1-\lambda)t}\ \D x^3-\sin{(1+\lambda)t}\ \D x^4,\\[4pt]
e^2=\ &\sin{(1-\lambda)t}\ \D x^1+\sin{(1+\lambda)t}\ \D x^2
-\cos{(1-\lambda)t}\ \D x^3-\cos{(1+\lambda)t}\ \D x^4,\\[4pt]
e^3=\ &-\sin{(1-\lambda)t}\ \D x^1+\sin{(1+\lambda)t}\ \D x^2
+\cos{(1-\lambda)t}\ \D x^3-\cos{(1+\lambda)t}\ \D x^4,\\[4pt]
e^4=\ &\cos{(1-\lambda)t}\ \D x^1+\cos{(1+\lambda)t}\ \D x^2
+\sin{(1-\lambda)t}\ \D x^3+\sin{(1+\lambda)t}\ \D x^4.
\end{split}
\]
Then the Sasaki-like metric is of the form
\begin{equation}\label{m2}
\begin{split}
g=\ & \D t^2%
-4\cos{2t}\left(\D x^1\D x^2-\D x^3\D x^4\right)-4\sin{2t}\left(\D
x^1\D x^4+\D x^2\D x^3\right).
\end{split}
\end{equation}
From \eqref{comstr2} it follows that the distribution
$H=%\ker(\eta)=
\Span\{e_1,\dots,e_4\}$ is integrable and the corresponding
integral submanifold  can be considered as
the holomorphic complex Riemannian
flat space $\R^{4}=\Span\{\D x^1,\dots,\D x^{4}\}$ with the
holomorphic complex Riemannian  structure  given by
\begin{equation*}
J\D x^1=\D x^3, \quad J\D x^2=\D x^4;\quad h=-4(\D x^1\D x^2-\D
x^3\D x^4), \quad \widetilde h =4(\D x^1\D x^4+\D x^2\D x^3)
\end{equation*}
and the Sasaki-like metric \eqref{m2} takes the form
$$g=\D t^2+\cos{2t}\
h-\sin{2t}\ \widetilde{h}.$$

\section{Curvature properties }%of the Sasaki-like almost contact complex Riemannian manifolds}

Let $\M$ be an almost contact complex Riemannian manifold. The
curva\-ture tensor of type $(1,3)$ is defined by
$R=[\nabla,\nabla]-\nabla_{[\ ,\ ]}$.  We denote the curvature
tensor of type $(0,4)$ by the same letter,
$R(x,y,z,u)=g(R(x,y)z,u)$. The Ricci tensor $Ric$,  the scalar
curvature $Scal$ and the *-scalar curvature $Scal^*$ are the usual
traces of the curvature,
$Ric(x,y)=\sum_{i=0}^{2n}\varepsilon_iR(e_i,x,y,e_i)$,
$Scal=\sum_{i=0}^{2n}\varepsilon_iRic(e_i,e_i)$,
$Scal^*=\sum_{i=0}^{2n}\varepsilon_i Ric(e_i,\f e_i)$.

\begin{prop}\label{vercurv}
On a Sasaki-like almost contact complex Riemannian manifold $\M$
the next formula holds
\begin{equation}\label{curf}
\begin{array}{l}
R(x,y,\f z,u)-R(x,y,z,\f u)\\[4pt]
=\left[g(y,z)-2\eta(y)\eta(z)\right]g(x,\f u)
+\left[g(y,u)-2\eta(y)\eta(u)\right]g(x,\f z)\\[4pt]
-\left[g(x,z)-2\eta(x)\eta(z)\right]g(y,\f
u)-\left[g(x,u)-2\eta(x)\eta(u)\right]g(y,\f z).
\end{array}
\end{equation}
In particular, we have
\begin{gather}
\label{cur}
 R(x,y)\xi=\eta(y)x-\eta(x)y, \quad  [X,\xi]\in H, \quad
\nabla_{\xi}X=-\f X-[X,\xi] \in H; \\[4pt]
R(\xi,X)\xi=-X, \qquad Ric(y,\xi)=2n\ \eta(y),\qquad
Ric(\xi,\xi)=2n.\label{ricxi}
\end{gather}
\end{prop}

\begin{proof}
The Ricci identity for $\f$ reads
\[
R(x,y,\f z,u)-R(x,y,z,\f
u)=g\Bigl(\left(\n_x\n_y\f\right)z,u\Bigr)-g\Bigl(\left(\n_y\n_x\f\right)z,u\Bigr).
\]
Applying \eqref{defsl} to the above equality and using
\eqref{sasaki1} we obtain \eqref{curf} by  straightforward
calculations. Set $z=\xi$ into \eqref{curf} and using \eqref{str1}
we get the first equality in \eqref{cur}. The  rest  follows from
\eqref{sasaki1} and the condition $\D \eta=0$. The equalities
\eqref{ricxi} follow directly from the first equality in
\eqref{cur}.
\end{proof}

\subsection{The horizontal curvature}
From $\D\eta=0$ it follows locally $\eta=\D x$, $H$ is integrable
and the manifold is locally the product
$M^{2n+1}=N^{2n}\times\mathbb R$ with $TN^{2n}=H$. The submanifold
$(N^{2n},J=\f|_H,h=g|_H)$ is a holomorphic complex
Riemannian manifold.
%(i.e. a K\"ahler-Norden manifold). %
Indeed, we obtain from \eqref{sasaki} that
$h(\nabla^h_X J)Y,Z)=F(X,Y,Z)=0$, where
$\nabla^h$ is the Levi-Civita connection of $h$.

We may consider $N^{2n}$ as a hypersurface of $M$ with the unit normal
$\xi=\frac{\partial}{\partial x}$. The equality \eqref{sasaki1}
yields
\[
g(\n_X\xi,Y)=-g(\n_XY,\xi)=-g(\f X,Y)=-\widetilde{g}|_{H}(X,Y),
\quad \n_{\xi}\xi=0.
\]
This means that the second fundamental form is equal to
$\widetilde{g}|_{H}=\widetilde h$. The Gauss equation (see e.g.
\cite[Chapter VII, Proposition~4.1]{KN}) yields
\begin{equation}\label{gaus}
R(X,Y,Z,U)=R^h(X,Y,Z,U)+g(\f X,Z)g(\f Y,U)-g(\f Y,Z)g(\f
X,U),
\end{equation}
where $R^h$ is the curvature tensor of the holomorphic complex
Riemannian manifold $(N^{2n},J,h)$.

For the horizontal Ricci tensor we obtain from \eqref{gaus} and \eqref{ricxi} that
\begin{multline}\label{ric}
Ric(Y,Z)=\sum_{i=1}^{2n}\varepsilon_iR(e_i,Y,Z,e_i)+R(\xi,Y,Z,\xi)\\
=Ric^h(Y,Z)+g(\f Y,\f Z)+g(Y,Z)=Ric^h(Y,Z),
\end{multline}
where $Ric^h$ is the Ricci tensor of $h=g|_{H}$.

It follows from \propref{vercurv}  that the curvature tensor in
the direction of $\xi$ on a Sasaki-like almost contact complex
Riemannian manifold is completely determined by
$\eta,\f,g,\widetilde g$. Indeed, using the properties of the
Riemannian curvature, we derive from \eqref{cur}
$$R(x,y,\xi,z)=R(\xi,z,x,y)=\eta(y)g(x,z)-\eta(x)g(y,z).$$
Now,
the equation \eqref{gaus} implies that the Riemannian curvature of
a Sasaki-like almost contact complex Riemannian manifold is
completely determined by the curvature of the underlying
holomorphic complex Riemannian manifold $(N^{2n}, TN^{2n}=H,J,h)$.

\subsection{Example~3: \protect{$S^1$}-solvable extension of the h-sphere}
The next example
illustrates Theorem~\ref{ext}. Consider $\R^{2n+2}$, $n>2$, as a
flat holomorphic complex Riemannian manifold, i.e.  $\R^{2n+2}$ is
equipped with the canonical complex structure $J'$ and the
canonical Norden metrics $h'$ and $\widetilde h'$  defined by
\[
\begin{aligned}
&h'(x',y')=\sum_{i=1}^{n+1} \left(x^i
y^i-x^{n+i+1}y^{n+i+1}\right), \quad \widetilde h'
(x',y')=-\sum_{i=1}^{n+1}\left(x^i
y^{n+i+1}+x^{n+i+1}y^i\right)%\\[4pt]
%&\widetilde h (x',y')=-\sum_{i=1}^{n+1}(x^i
%y^{n+i+1}+x^{n+i+1}y^i)
\end{aligned}
\]

for the vectors $x' = (x^1, \dots, x^{2n+2})$ and $y' = (y^1,
\dots, y^{2n+2})$ %with respect to the basis \{e_1
in $\R^{2n+2}$. Identifying the point $z' = (z^1, \dots,
z^{2n+2})$ in $\R^{2n+2}$ with the position vector $Z'$, we
consider the  complex hypersurface $S_h^{2n}(z'_0; a, b)$
defined by the equations
\[
h'(z'-z'_0, z'-z'_0) = a,\quad \widetilde h' (z'-z'_0, z'-z'_0) =
b,
\]
where $(0, 0)\allowbreak\neq\allowbreak (a, b) \in\R^2$. The
co-dimension two submanifold $S_h^{2n}(z'_0; a,b)$ is
$J'$-invariant and the restriction of $h'$ on $S_h^{2n}(z'_0;
a,b)$ has rank $2n$ due to the condition $(0,
0)\allowbreak\neq\allowbreak (a, b)$. The holomorphic complex
Riemannian structure on $\R^{2n+2}$ inherits a holomorphic complex
Riemannian structure $\bigl(J'|_{{S_h^{2n}}},
h'|_{S_h^{2n}}\bigr)$ on the complex hypersurface $S_h^{2n}(z'_0;
a,b)$. The holomorphic complex Riemannian manifold
$\bigl(S_h^{2n}(z'_0; a,b),\allowbreak{}
J'|_{S_h^{2n}},\allowbreak{} h'|_{S_h^{2n}}\bigr)$ is sometimes
called an \emph{h-sphere} with center $z'_0$ and a pair of
parameters $(a, b)$. The h-sphere $S_h^{2n}(z'_0; 1, 0)$ is the
sphere of Kotel'nikov-Study \cite{v}. The curvature of an h-sphere
is given by the formula \cite{GGM0}
\begin{equation}\label{Rnu}
R'|_{S_h^{2n}} = \frac{1}{a^2+b^2} \bigl\{a(\pi_1 - \pi_2) - b
\pi_3\bigr\},
\end{equation}
where $\pi_1 = \frac{1}{2}h'|_{S_h^{2n}}\owedge h'|_{S_h^{2n}}$,
$\pi_2 = \frac{1}{2}\widetilde{h}'|_{S_h^{2n}}\owedge
\widetilde{h}'|_{S_h^{2n}}$, $\pi_3 =-
h'|_{S_h^{2n}}\owedge\widetilde{h}'|_{S_h^{2n}}$ and $\owedge$
stands for the Kulkarni-Nomizu product of two (0,2)-tensors; for
example,
\[
\begin{aligned}
h'\owedge\widetilde{h}'(X,Y,Z,U)&=h'(Y,Z) \widetilde h'(X,U) -
h'(X,Z)\widetilde h'(Y,U)\\[4pt] &+\widetilde h'(Y,Z) h'(X,U) - \widetilde
h'(X,Z) h'(Y,U).
\end{aligned}
\]
Consequently, we have
\begin{equation}\label{Ric-h-sph}
Ric'|_{S_h^{2n}}=\frac{2(n-1)}{a^2+b^2}\bigl(a
h'|_{S_h^{2n}}+b\widetilde{h}'|_{S_h^{2n}}\bigr), \qquad
Scal'|_{S_h^{2n}}=\frac{4n(n-1)a}{a^2+b^2}.
\end{equation}

The  product manifold $M^{2n+1}=\mathbb R^+\times S_h^{2n}(z'_0;
a,b)$  equipped with the following  almost contact complex
Riemannian structure
\begin{equation*}
\eta=\D t, \quad \varphi|_H =J'|_{S_h^{2n}}, \quad
\eta\circ\varphi =0,\quad g=\D t^2+\cos{2t}\
h'|_{S_h^{2n}}-\sin{2t}\ \widetilde h'|_{S_h^{2n}}
\end{equation*}
is a Sasaki-like almost contact complex Riemannian manifold according to \thmref{ext}.

The horizontal metrics on $M^{2n+1}=\mathbb R^+\times S_h^{2n}(z'_0;
a,b)$  are
\begin{equation}\label{hh}
\begin{array}{l}
h=g|_H=\cos{2t}\ h'|_{S_h^{2n}}-\sin{2t}\ \widetilde
h'|_{S_h^{2n}},\\[4pt]
\widetilde h=\widetilde g|_H=\sin{2t}\ h'|_{S_h^{2n}}+\cos{2t}\
\widetilde h'|_{S_h^{2n}}.
\end{array}
\end{equation}
The Levi-Civita connection $\nabla'$ of the metric
$h'|_{S_h^{2n}}$ coincides with the Levi-Civita connection of
$\widetilde{h}'|_{S_h^{2n}}$ since $\nabla' J'=0$. Using this
fact,  the Koszul formula \eqref{koszul} together with \eqref{hh}
gives for the Levi-Civita connection $\nabla^h$ of $h$ the
expression
\[
h(\nabla^h_XY,Z)= \cos{2t}\
h'|_{S_h^{2n}}\left(\nabla'_XY,Z\right)-\sin{2t}\
h'|_{S_h^{2n}}\left(\nabla'_XY,JZ\right)
=h\left(\nabla'_XY,Z\right),
\]
which implies $\nabla^h_XY=\nabla'_XY$. The latter equality
together with \eqref{hh} yields for the curvature of $h$ the
formula $R^h=\cos{2t}\ R'|_{S_h^{2n}}-\sin{2t}\
\widetilde{R}'|_{S_h^{2n}}$, where
$\widetilde{R}'|_{S_h^{2n}}:=J'R'|_{S_h^{2n}}$. The above equality
together with \eqref{Rnu} implies
\begin{equation}\label{rrr}
\begin{split}
R^h&=\frac{1}{a^2+b^2}\bigl\{\cos{2t}[a(\pi_1-\pi_2)-b\pi_3]
-\sin{2t}[-a\pi_3-b(\pi_1-\pi_2)]\bigr\}\\[4pt]
&=\frac1{a^2+b^2}\bigl\{(a\cos{2t}+b\sin{2t})(\pi_1-\pi_2)
-(b\cos{2t}-a\sin{2t})\pi_3\bigr\}.
\end{split}
\end{equation}

We obtain from \eqref{gaus}, \eqref{hh} and \eqref{rrr} that the
horizontal curvature $R|_H$ of the Sasaki-like almost contact
complex Riemannian manifold $M^{2n+1}=\mathbb R^+\times
S_h^{2n}(z'_0; a,b)$ is given by the formula
\[
\begin{split}
R|_H&=R^h-(\sin{2t})^2\pi_1-(\cos{2t})^2\pi_2+\sin{2t}\cos{2t}\pi_3\\[4pt]
&=\left\{\frac1{a^2+b^2}(a\cos{2t}+b\sin{2t})-(\sin{2t})^2\right\}\pi_1\\[4pt]
&-\left\{\frac1{a^2+b^2}(a\cos{2t}+b\sin{2t})+(\cos{2t})^2\right\}\pi_2\\[4pt]
&-\left\{\frac1{a^2+b^2}(b\cos{2t}-a\sin{2t})-\sin{2t}\cos{2t}\right\}\pi_3.
\end{split}
\]
For the horizontal Ricci tensor we obtain from \eqref{ric},
\eqref{Ric-h-sph} and \eqref{hh} the formula
\[
\begin{split}
Ric|_H=Ric^h&=\frac{2(n-1)}{a^2+b^2}(a
h'|_{S_h^{2n}}+b\widetilde{h}'|_{S_h^{2n}})\\[4pt]
&=\frac{2(n-1)}{a^2+b^2}\Big[(a\cos{2t}-b\sin{2t})h+(b\cos{2t}+a\sin{2t})\widetilde
h\Big].
\end{split}
\]

\section{Contact conformal (homothetic) transformations}% of the Sasaki-like spaces}

In this section we investigate when the Sasaki-like condition is
preserved under contact conformal transformations. We recall that
a general contact conformal transformation of an almost contact
complex Riemannian manifold $\M$ is defined by
\cite{Man4,ManGri1,ManGri2}
\begin{equation}\label{cct}
\begin{aligned}
\overline{\eta}&=e^w\eta,\quad \overline{\xi}=e^{-w}\xi,\\[4pt]
\overline{g}(x,y)& = e^{2u}\cos{2v}\ g(x,y)+e^{2u}\sin{2v}\ g(x,\f
y)+(e^{2w}-e^{2u}\cos{2v})\eta(x)\eta(y),
\end{aligned}
\end{equation}
where $u$, $v$, $w$ are smooth functions.

If the functions $u$, $v$, $w$ are constant we have a
\emph{contact homothetic transformation}.

The tensors $\overline F$ and $F$ are connected by \cite{Man4}, see also \cite[(22)]{MI},
\begin{equation}\label{ff}
\begin{aligned}
    2\overline{F}&(x,y,z)=2e^{2u}\cos{2v} F(x,y,z)
    +2e^{2w}\eta(x)\left[\eta(y)\D w(\f z)+\eta(z)\D w(\f
    y)\right]\\[4pt]
    &
    +e^{2u}\sin{2v} \left[F(\f y,z,x)-F(y,\f z,x)+F(x,\f y,\xi)\eta(z)\right]\\[4pt]
    &+e^{2u}\sin{2v}
    \left[F(\f z,y,x)-F(z,\f y,x)+F(x,\f
    z,\xi)\eta(y)\right]\\[4pt]
    &+(e^{2w}-e^{2u}\cos{2v}\left[F(x,y,\xi)+F(\f y,\f
    x,\xi)\right]\eta(z)\\[4pt]
    &+(e^{2w}-e^{2u}\cos{2v})
    \left[F(x,z,\xi)+F(\f z,\f x,\xi)\right]\eta(y)\\[4pt]
    &+(e^{2w}-e^{2u}\cos{2v})
    \left[F(y,z,\xi)+F(\f z,\f y,\xi)\right]\eta(x)\\[4pt]
    &+(e^{2w}-e^{2u}\cos{2v})\left[F(z,y,\xi)+F(\f y,\f z,\xi)\right]\eta(x)
\\[4pt]
    &-2e^{2u}
    \cos{2v}\left[\D u(\f z)+\D v(z)\right]
    -2e^{2u}\sin{2v}\left[\D u(z)-\D v(\f z)\right]g(\f x,\f y)
\\[4pt]
    &-2e^{2u}
    \cos{2v}\left[\D u(\f y)+\D v(y)\right]
    -2e^{2u}\sin{2v}\left[\D u(y)-\D v(\f y)\right]g(\f x,\f z)
\\[4pt]
    &-2e^{2u}
    \cos{2v}\left[\D u(z)-\D v(\f z)\right]
    +2e^{2u}\sin{2v}\left[\D u(\f z)+\D v(z)\right]g(x,\f y)
\\[4pt]
    &-2e^{2u}
    \cos{2v}\left[\D u(y)-\D v(\f y)\right]
    +2e^{2u}\sin{2v}\left[\D u(\f y)+\D v(y)\right]g(x,\f z).
\end{aligned}
\end{equation}

The Sasaki-like condition  \eqref{defsl} also reads as
\begin{equation}\label{defsl1}
F(x,y,z)=g(\f x,\f y)\eta(z)+g(\f x,\f z)\eta(y).
\end{equation}
We obtain the Sasaki-like condition for the metric $\overline g$
substituting \eqref{cct} into \eqref{defsl1} which yields
\begin{equation}\label{defbar}
\begin{aligned}
\overline F(x,y,z)=e^{w+2u}&\Big\{\cos{2v}\bigl[\eta(z)g(\f x,\f
y)+\eta(y)g(\f x,\f z)\bigr]\\&-\sin{2v}\bigl[\eta(z)g(x,\f
y)+\eta(y)g(x,\f z)\bigr]\Big\}.
\end{aligned}
\end{equation}
Substitute \eqref{defsl1} into \eqref{ff} to get
\begin{subequations}\label{ff1}
\begin{equation}
\begin{split}
   \overline{F} (x,y,z)&=e^{2w}\eta(x)\left\{\eta(y)\D w(\f z)+\eta(z)\D w(\f y)\right\}
\\[4pt]
    &+e^{2u}\Bigl\{
    \cos{2v}\left[\eta(z)g(\f x,\f y)+\eta(y)g(\f x,\f z)\right]\\[4pt]
    &\phantom{=e^{2u}\Bigl\{}
    -\sin{2v}\left[\eta(z)g(x,\f y)+\eta(y)g(x,\f z)\right]\Bigr\}
\\[4pt]
\end{split}
\end{equation}
\begin{equation}
\begin{split}
    &-e^{2u}\Bigl\{%
    \left\{%
    \cos{2v}\left[\D u(\f z)+\D v(z)\right]
    +\sin{2v}\left[\D u(z)-\D v(\f z)\right]\right\}g(\f x,\f y)
\\[4pt]
    &\phantom{-e^{2u}\Bigl\{}%
    +\left\{%
    \cos{2v}\left[\D u(\f y)+\D v(y)\right]
    +\sin{2v}\left[\D u(y)-\D v(\f y)\right]\right\}g(\f x,\f z)
\\[4pt]
    &\phantom{-e^{2u}\Bigl\{}%
    +\left\{%
    \cos{2v}\left[\D u(z)-\D v(\f z)\right]
    -\sin{2v}\left[\D u(\f z)+\D v(z)\right]\right\}g(x,\f y)
\\[4pt]
    &\phantom{-e^{2u}\Bigl\{}%
    +\left\{%
    \cos{2v}\left[\D u(y)-\D v(\f y)\right]
    -\sin{2v}\left[\D u(\f y)+\D v(y)\right]\right\}g(x,\f z)
    \Bigr\}.
\end{split}
\end{equation}
\end{subequations}
The equalities \eqref{ff1} and \eqref{defbar} imply
\begin{equation}\label{ff2}
\begin{split}
    &(1-e^w)e^{2u}\Bigl\{
    \cos{2v}\left[\eta(z)g(\f x,\f y)+\eta(y)g(\f x,\f z)\right]\\[4pt]
    &\phantom{(1-e^w)e^{2u}\Bigl\{}
    -\sin{2v}\left[\eta(z)g(x,\f y)+\eta(y)g(x,\f z)\right]\Bigr\}
\\[4pt]
    &-e^{2u}\Bigl\{%
    \left\{%
    \cos{2v}\left[\D u(\f z)+\D v(z)\right]
    +\sin{2v}\left[\D u(z)-\D v(\f z)\right]\right\}g(\f x,\f y)
\\[4pt]
    &\phantom{-e^{2u}\Bigl\{}%
    +\left\{%
    \cos{2v}\left[\D u(\f y)+\D v(y)\right]
    +\sin{2v}\left[\D u(y)-\D v(\f y)\right]\right\}g(\f x,\f z)
\\[4pt]
    &\phantom{-e^{2u}\Bigl\{}%
    +\left\{%
    \cos{2v}\left[\D u(z)-\D v(\f z)\right]
    -\sin{2v}\left[\D u(\f z)+\D v(z)\right]\right\}g(x,\f y)
\\[4pt]
    &\phantom{-e^{2u}\Bigl\{}%
    +\left\{%
    \cos{2v}\left[\D u(y)-\D v(\f y)\right]
    -\sin{2v}\left[\D u(\f y)+\D v(y)\right]\right\}g(x,\f z)
    \Bigr\}
\\[4pt]
    &%
    +e^{2w}\eta(x)\left\{\eta(y)\D w(\f z)+\eta(z)\D w(\f y)\right\}
    =0.
\end{split}
\end{equation}
Set $x=y=\xi$ into \eqref{ff2} to get
\begin{equation}\label{www}
\D w(\f z)=0.
\end{equation}
 Now, using \eqref{www} we write
\eqref{ff2} in  the form
\begin{equation}\label{ff3}
A(z)g(\f x,\f y)+B(z)g(x,\f y)+A(y)g(\f x,\f z)+B(y)g(x,\f z)=0,
\end{equation}
where the 1-forms $A$ and $B$ are defined by
\begin{equation}\label{ggg}
\begin{split}
A(z)=\cos{2v}\left[(e^w-1)\eta(z)+\D u(\f z)+\D v(z)\right]
    +\sin{2v}\left[\D u(z)-\D v(\f z)\right],\\[4pt]
B(z)=\sin{2v}\left[(e^w-1)\eta(z)+\D u(\f z)+\D v(z)\right]
    -\cos{2v}\left[\D u(z)-\D v(\f z)\right].
\end{split}
\end{equation}
Taking the trace of \eqref{ff3} with respect to $x=e_i$, $z=e_i$
and $y= e_i$, $z=e_i$ to get
\begin{equation}\label{ff4}
 - (2n+1)A(z)+\eta(z)A(\xi)
+B(\f z)=0, \quad
    A(z)-\eta(z)A(\xi) -B(\f z)=0.
\end{equation}
We derive from \eqref{ff4} that $A(z)=0$. Similarly, we obtain
$B(z)=0$. Now, \eqref{ggg} imply
\begin{equation}\label{ggg1}
\begin{split}
&\cos{2v}\left[\D u(\f z)+\D v(z)\right]
    +\sin{2v}\left[\D u(z)-\D v(\f z)\right]=(1-e^w)\cos{2v}\ \eta(z),\\[4pt]
&\sin{2v}\left[\D u(\f z)+\D v(z)\right]
    -\cos{2v}\left[\D u(z)-\D v(\f z)\right]=(1-e^w)\sin{2v}\ \eta(z).
\end{split}
\end{equation}

Comparing \eqref{defsl1} and \eqref{ff} we derive
\begin{prop}\label{prop:Sasaki}
Let $\M$ be a Sasaki-like almost contact complex Rie\-mann\-ian
manifold. Then the structure $(\f,
\overline{\xi},\overline{\eta},\overline g)$ defined by
\eqref{cct} is Sasaki-like if and only if the smooth functions
$u,v,w$ satisfy the following conditions
\begin{equation}\label{sssl}
dw\circ\f=0,\quad \D u-\D v\circ\f=0,\quad \D u\circ\f+\D v=(1-e^w)\eta.
\end{equation}
In particular
\[
\D u(\xi)=0,\qquad \D v(\xi)=1-e^w.
\]
In the case $w=0$ the global smooth functions $u$ and $v$ does not
depend on $\xi$ and are locally defined on the complex submanifold
$N^{2n}$, $TN=H$, and the complex valued function $u+\sqrt{-1}v$
is a holomorphic function on $N^{2n}$.
\end{prop}
\begin{proof}
%Substitute $\alpha=e^{2u}\cos{2v}, \beta=e^{2v}\sin{2v}$ into
%\eqref{ggg1} and
Solve the %obtained
linear system \eqref{ggg1} to get the second and the third
equality into \eqref{sssl}. Now, \eqref{www} completes the proof
of \eqref{sssl}.
\end{proof}

\subsection{Contact homothetic transformations}
Let us consider contact homothetic transformations of an almost
contact complex Riemannian manifold $\M$. Since the functions $u$,
$v$, $w$ are constant, it follows from \eqref{cct} using the
Koszul formula \eqref{koszul} that the Levi-Civita connections
$\overline{\n}$ and $\n$ of the metrics $\overline g$ and $g$,
respectively, are connected by the formula
\begin{equation}\label{barl}
\overline{\n}_xy=\n_xy+e^{2(u-w)}\sin{2v}\ g(\f x,\f y)\xi
-\left(e^{-2w}-e^{2(u-w)}\cos{2v}\right) g(x,\f y)\xi.
\end{equation}

For the corresponding curvature tensors $\overline R$ and $R$   we
obtain from \eqref{barl} that
\begin{equation}\label{barRR}
\begin{aligned}
\overline{R}(x,y)z&={R}(x,y)z\\[4pt]
&+e^{2(u-w)}\sin{2v} %
\left\{g(y,\f z)\eta(x)\xi-g(\f y,\f z)\f x\right.\\[4pt]
&\phantom{+e^{2(u-w)}\sin{2v}}\left.-g(x,\f z)\eta(y)\xi+g(\f x,\f
z)\f
y\right\}\\[4pt]
 &+\left(e^{-2w}-e^{2(u-w)}\cos{2v}\right)\left\{g(\f y,\f
z)\eta(x)\xi+g(y,\f z)\f x\right.\\[4pt]
&\phantom{+\left(e^{-2w}-e^{2(u-w)}\cos{2v}\right)}\left.-g(\f
x,\f z)\eta(y)\xi-g(x,\f z)\f y\right\}.
\end{aligned}
\end{equation}
We have
\begin{prop}\label{homric}
The Ricci tensor of an almost contact complex Riemannian manifold
is invariant under a contact homothetic transformation,
\begin{equation}\label{ri}
\overline{Ric}=Ric.
\end{equation}
Consequently, we obtain
\begin{equation}\label{scal}
\begin{array}{l}
\overline{Scal}=e^{-2u}\cos{2v}\ Scal-e^{-2u}\sin{2v}\ Scal^*
+\left(e^{-2w}-e^{-2u}\cos{2v}\right)Ric(\xi,\xi),\\[4pt]
\overline{Scal}^*=e^{-2u}\sin{2v}\ Scal+e^{-2u}\cos{2v}\ Scal^*
-e^{-2u}\sin{2v}\ Ric(\xi,\xi).
\end{array}
\end{equation}
In particular, the scalar curvatures of a Sasaki-like almost
contact complex Riemannian manifold changes under a contact
homothetic transformation with $w=0$ as follows
\begin{equation}\label{scalsas}
\begin{array}{l}
\overline{Scal}=e^{-2u}\cos{2v}\ Scal-e^{-2u}\sin{2v}\ Scal^*
+2n\left(1-e^{-2u}\cos{2v}\right),\\[4pt]
\overline{Scal}^*=e^{-2u}\sin{2v}\ Scal+e^{-2u}\cos{2v}\ Scal^*
-2n\ e^{-2u}\sin{2v}.
\end{array}
\end{equation}
\end{prop}
\begin{proof}
Taking the trace of \eqref{barRR} we get $\overline{Ric}=Ric$.

Consider the basis $\{\overline e_0=\overline\xi$, $\overline
e_1$, $\dots$, $\overline e_{n}$, $\overline e_{n+1}=\f \overline
e_1$, $\dots$, $\overline e_{2n}=\f \overline e_n\}$, where
\[
\overline e_i=e^{-u}\left\{\cos v\ e_i-\sin v\ \f
e_i\right\},\quad i=1,\dots, n.
\]
It is easy to check that this basis is orthonormal for the metric
$\overline g$. Then \eqref{ri} gives
$\overline{Scal}=\sum_{i=0}^{2n}\overline\varepsilon_i \overline
Ric(\overline e_i,\overline e_i)$ and
$\overline{Scal}^*=\sum_{i=0}^{2n}\overline\varepsilon_i \overline
Ric(\overline e_i,\f \overline e_i)$ which yield the formulas for
the scalar curvatures.

The formulas \eqref{scalsas} follow from \eqref{scal} and
\eqref{ricxi}.
\end{proof}
Consequently, we have
\begin{prop}
A Sasaki-like almost contact complex Riemannian manifold $\M$ is
Einstein if and only if the underlying holomorphic complex
Riemannian manifold $(N^{2n}, TN^{2n}=H,J,h)$  is an Einstein
manifold with scalar curvature not depending on the vertical
direction $\xi$.
\end{prop}
\begin{proof}
Compare \eqref{ricxi} with \eqref{ric} to see that $\M$ is an
Einstein manifold if and only if $N$ is an Einstein manifold with
Einstein constant equal to $2n$, $Ric^h=2n\,g$. Further, consider
a contact homothetic transformation with $w=v=0$ we get that
$\bigl(M,\f,\xi,\eta,\overline
g=e^{2u}g+(1-e^{2u})\eta\otimes\eta\bigr)$ is again Sasaki-like
due to Proposition~\ref{prop:Sasaki}. Applying
Proposition~\ref{homric} and \eqref{ric}, we get the following
sequence of equalities
\[
\overline{Ric}^{\bar h}=\overline{Ric}{|_H}=Ric|_H=Ric^h
=\frac{Scal}{2n}g|_H=\frac{e^{-2u}Scal^h}{2n}\overline{g}|_H,
\]
which yield $\overline{Scal}^{\bar h}=e^{-2u}Scal^h=4n^2$ by
choosing the constant $u$ to be equal to
$e^{-2u}=\frac{4n^2}{Scal^h}$, i.e. the Einstein constant of the
complex holomorphic Einstein manifold $N$ can always be made equal
to $4n^2$ which completes the proof.
\end{proof}
Suppose we have a Sasaki-like almost contact complex Riemannian
manifold which is Einstein, $Ric=2n\,g$, and make a contact
homothetic transformation
\begin{equation}\label{cctt}
\overline{\eta}=\eta,\quad \overline{\xi}=\xi,\quad
\overline{g}(x,y) = c\ g(x,y)+d\ g(x,\f y)+(1-c)\eta(x)\eta(y),
\end{equation}
where $c$, $d$ are constants. According to
Proposition~\ref{homric} we obtain using \eqref{cctt} that
\begin{equation}\label{etaein}
\begin{split}
\overline{Ric}(x,y)&=Ric(x,y)=2n\,g(x,y)
\\[4pt]
&=\frac{2n}{c^2+d^2}\bigl\{c\,\overline g(x,y) - d\,\overline
g(x,\f y) +(c^2+d^2-c)\eta(x)\eta(y)\bigr.\}.
\end{split}
\end{equation}
We may call a Sasaki-like space whose Ricci tensor satisfies
\eqref{etaein}  \emph{an $\eta$-complex-Ein\-stein Sasaki-like
manifold} and if the constant $d$ vanishes, $d=0$, we have
\emph{$\eta$-Ein\-stein} Sasaki-like space. Thus, we have shown
\begin{prop}
Any $\eta$-complex-Einstein Sasaki-like space is contact
homothetic to an Einstein Sasaki-like space.
\end{prop}

\section*{Acknowledgments}
S.I. is  partially supported by the Contract 156/2013 with the
University of Sofia ,,St. Kliment Ohridski``. M.M. and H.M. are
partially supported by the project NI13-FMI-002 of the Scientific
Research Fund at the University of Plovdiv. Authors are partially
supported by the project ,,Center of Excellence for Applications
of Mathematics`` of the German Academic Exchange Service (DAAD).

The authors would like to thank Professor Marisa Fernandez for
very useful comments.

%%%%%%%%%%%%%%%%%%%%%%%%%%%%%%%%%%%%%%%%%%%%%%%%%%%%%%%%%%%%%%%%%%%%%%
\end{document}